\documentclass[11pt]{article}
\usepackage{amsmath, amsthm, amscd, amsfonts, amssymb}
\date{ }

\newcommand{\ga}{\Gamma}

\newcommand{\mat}{\mathbb}

\newcommand{\al}{\alpha}

\newtheorem{theorem}{Theorem}[section]
\newtheorem{lemma}[theorem]{Lemma}
\newtheorem{problem}[theorem]{Problem}

\title{\bf On some Frobenius groups with the same prime graph as the almost simple group ${\rm {\bf PGL(2,49)}}$}
\smallskip
\normalsize
\author{{\bf Ali Mahmoudifar}\\
Department of Mathematics, Tehran-North Branch,\\ Islamic Azad University, Tehran, Iran\\
e-mail: alimahmoudifar@gmail.com}
\begin{document}
\maketitle
\begin{abstract}
The prime graph of a finite group $G$ is denoted by
$\ga(G)$ whose vertex set is $\pi(G)$ and two distinct primes $p$ and $q$ are adjacent in $\ga(G)$, whenever $G$ contains an element with order $pq$. We say that $G$ is unrecognizable by prime graph if there is a finite group $H$ with $\ga(H)=\ga(G)$, in while $H\not\cong G$. In this paper, we consider finite groups with the same prime graph as the almost simple group $\textrm{PGL}(2,49)$. Moreover, we construct some Frobenius groups
whose their prime graph coincide with $\ga(\textrm{PGL}(2,49))$, in particular, we get that $\textrm{PGL}(2,49)$ is unrecognizable by prime graph.
\end{abstract}
\section{Introduction}
Let $\mat{N}$ denotes the set of natural numbers. If $n\in \mat{N}$,
then we denote by $\pi(n)$, the set of all prime divisors of $n$.
Let $G$ be a finite group. The set $\pi(|G|)$ is denoted by
$\pi(G)$. Also the set of element orders of $G$ is denoted by
$\pi_e(G)$. We denote by $\mu(S)$, the maximal numbers of $\pi_e(G)$ under the divisibility
relation. The \textit{prime graph} of $G$ is a graph whose vertex
set is $\pi(G)$ and two distinct primes $p$ and $q$ are joined by an
edge (and we write $p\sim q$), whenever $G$ contains an element of
order $pq$. The prime graph of $G$ is denoted by $\ga(G)$.
A finite group $G$ is called
\textit{unrecognizable by prime graph} if for every finite group $H$ such that $\ga(H) = \ga(G)$, however
$H\ncong G$.

In \cite{15}, it is proved that if
$p$ is a prime number which is not a Mersenne or Fermat prime and $p \not= 11$, $19$ and
$\ga(G) = \ga(\textrm{PGL}(2, p))$, then $G$ has a unique nonabelian composition factor which is
isomorphic to $\textrm{PSL}(2, p)$ and if $p = 13$, then $G$ has a unique nonabelian composition
factor which is isomorphic to $\textrm{PSL}(2,13)$ or $\textrm{PSL}(2,27)$.
In \cite{pgl-spec}, it is proved that if $q = p^{\al}$, where $p$ is a prime and $\al > 1$, then $\textrm{PGL}(2, q)$
is uniquely determined by its element orders. Also in \cite{pgl}, it is proved that if $q = p^{\al}$, where $p$ is an odd prime and $\al$ is an odd natural number, then $\textrm{PGL}(2, q)$ is uniquely determined by its prime graph.
However, in this paper as the main result we prove that the almost simple group $\textrm{PGL}(2,49)$  is unrecognizable by prime graph. Also, finally we put a question about the existence of Frobenius groups with the same prime graph as the almost simple groups $\textrm{PGL}(2, q)$. 
\section{Preliminary Results}
\begin{lemma}\label{maza}(\cite{maza})
Let $G$ be a finite group and $N \unlhd G$ such that $G/N$ is a
Frobenius group with kernel $F$ and cyclic complement $C$. If
$(|F|,|N|)=1$ and $F$ is not contained in $NC_G(N)/N$, then
$p|C|\in\pi_e(G)$ for some prime divisor $p$ of $|N|$.
\end{lemma}
\begin{lemma}\label{hig}(\cite{hig})
Let $G$ be a finite group and $|\pi(G)|\geq3$.
If there exist prime numbers $r$, $s$, $t \in\pi(G)$, such that $\{tr, ts, rs\} \cap\pi_e(G) =\emptyset$,
then $G$ is non-solvable.
\end{lemma}
\begin{lemma}\label{fro}(\cite[Theorem 18.6]{pas})
Let $G$ be a nonsolvable Frobenius complement. Then $G$ has a normal subgroup $G_0$ with $|G:G_0|=1$ or $2$ such that $G_0={\rm SL}(2,5)\times M$ with $M$ a Z-group of order prime to $2$, $3$ and $5$.
\end{lemma}
Using \cite[Theorem A]{21}, we have the following result:
\begin{lemma}\label{wil}
Let $G$ be a finite group with $t(G)\geq 2$. Then
one of the following holds:

(a) $G$ is a Frobenius or 2-Frobenius group;

(b) there exists a nonabelian simple group $S$ such that $S\leq$
$\overline{G}$$:=G/N \leq \textrm{Aut}(S)$ for some nilpotent normal
subgroup $N$ of $G$.
\end{lemma}
\begin{lemma}\label{zavar}(\cite{zavar})
Let $G = L_n^{\varepsilon}(q)$, $q = p^m$, be a simple group which acts absolutely
irreducibly on a vector space $W$ over a field of characteristic $p$. Denote $H = W\leftthreetimes G$. If $n = 2$ and $q$ is odd then $2p\in\pi_e(H)$.
\end{lemma}
\section{Main Results}
\begin{lemma}\label{l2}
There are infinitely many finite Frobenius group $G$ such that $\ga(G)=\ga({\rm PGL}(2,49))$.
\end{lemma}
\begin{proof}
Let $F$ be a finite field of characteristic $7$. Also let there are some elements $\alpha$ and $\beta$ included in $F$ such that $\alpha^2=-1$ and $\beta^2=5$. We know that such a finite filed exists and moreover there are infinitely many filed with these properties.

Now we construct some linear groups as follow:
\begin{equation*}
C := \langle
\left(\begin{array}{ccc}
-1 & 1 & 0 \\
-1 & 0 & 0 \\
0 & 0 & 1
\end{array} \right),
\left(\begin{array}{ccc}
0 & \alpha  &0\\
\alpha & \frac{\beta+1}{2} & 0 \\
0 & 0 & 1
\end{array}\right),
\left(\begin{array}{ccc}
-1 & 0  &0\\
0 & -1 & 0 \\
0 & 0 & 1
\end{array}\right)
\rangle,
\end{equation*}

\begin{equation*}
K := \langle
\left(
\begin{array}{ccc}
0 & 0 & 1 \\
0 & 0 & 0 \\
0 & 0 & 1
\end{array} \right),
\left(
\begin{array}{ccc}
0 & 0 & 0 \\
0 & 0 & 1 \\
0 & 0 & 1
\end{array} \right)
\rangle.
\end{equation*}
By the above definition, $C\cong\langle x,y,z|x^3=y^5=z^2=1, x^z=z, y^z=y, (xy)^2=z\rangle$. This implies that $C\cong {\rm SL}(2,5)$. Also we have $K\cong F\oplus F$, is a direct sum of additive group $F$ by itself. This means $K$ is isomorphic to a vector space of dimension $2$ over $F$ and so $|K|=|F|^2$. It is obvious that $C$ belongs to the normalizer of $K$ in ${\rm GL}(3,F)$.

Now we define $G:=K\rtimes C$. Since $K$ is an elementary abelian $7$-group, it is easy to prove that $C$ acts fixed point freely on $K$ by conjugation. Hence $G$ is a Frobenius group with kernel $K$ and complement $C$. This implies that in the prime graph of $G$, $7$ is an isolated vertex. Also by $\ga({\rm SL}(2,5))$, we get that $2$ is adjacent to $3$ and $5$ and there is no edge between $3$ and $5$ in $\ga(G)$. Therefore, $\ga(G)$ coincides to $\ga({\rm PGL}(2,49))$, which completes the proof.
\end{proof}
\begin{lemma}\label{l3}
Each following group $G$ is an almost simple group related to the simple group $S$. Moreover, $G$ has a prime graph which coincide with the prime graph of the almost simple group ${\rm PGL}(2,49)$:

(1) $G=S_7$ and $S=A_7$.

(2) $G=U_4(3)\cdot 2$ and $S=U_4(3)$.

(3) $G=U_3(5)$ or $G=U_3(5)\cdot2$ and $S=U_3(5)$.
\end{lemma}
\begin{proof}
Using \cite{atlas}, it is straightforward.
\end{proof}
\begin{theorem}\label{tm}
Let $G$ be a finite group with the prime graph as same as the prime graph of $\emph{PGL}(2,49)$. Then $G$ is isomorphic to
one of the following groups:

(1) A Frobenius group $K\rtimes C$, such that $K$ is a $7$-group and $C$ contains a subgroup $C_0$ whose index in $C$ is at most $2$ and $C_0$ is isomorphic to ${\rm SL}(2,5)$.

(2) On of the almost simple group: $S_7$, $U_4(3)\cdot 2$, $U_3(5)\cdot2$ or ${\rm PGL}(2,49)$.

(3) The simple group: $U_3(5)$.

In particular, ${\rm PGL}(2,49)$ is unrecognizable by prime graph.
\end{theorem}
\begin{proof}
By \cite[Lemma 7]{mogh-shi}, it follows that $\mu(\textrm{PGL}(2, 49)) =\{7, 48, 50\}$.
Hence, the connected components of the prime graph of $\textrm{PGL}(2, 49)$ are exactly $\{7\}$ and $\{2,3,5\}$. Also by
$\mu(\textrm{PGL}(2, 49))$, there is no edge between $3$ and $5$ in $\ga(\textrm{PGL}(2, 49))$. Now since $\ga(G)=\ga($\textrm{PGL}(2, 49)$)$, we deduce that these relations hold in the prime graph of $G$.

First we claim that $G$ is not solvable. On the contrary, let $G$ be a solvable group. So there is a Hall $\{3,5,7\}$-subgroup in $G$, say $H$. On the other hand $\{3,5,7\}$ is an independent subset of $\ga(G)$, which is a contradiction by Lemma~\ref{hig}. Therefore, $G$ is not solvable and so by Lemma~\ref{wil}, either $G$ is a Frobenius group or there is a nonabelian simple group $S$ such that
$S\leq \bar{G}:=G/{\rm Fit}(G)\leq {\rm Aut}(S)$.

Let $G$ be a Frobenius group with kernel $K$ and complement $C$. By Lemma~\ref{fro}, we know that $K$ is nilpotent and $\pi(C)$ is a connected component of the prime graph of $G$. Hence we conclude that $\pi(K)=\{7\}$ and $\pi(C)=\{2,3,5\}$, since $7$ is an isolated vertex in $\ga(G)$.
Hence if $C$ is solvable, then $G$ is a solvable which is a contradiction by the above argument.

Thus we suppose that $C$ is non-solvable. Then by Lemma~\ref{fro}, the complement $C$ has a normal subgroup $C_0$ with index at most $2$ which is isomorphic to ${\rm SL}(2,5)\times M$, where $\pi(M)\cap\{2,3,5\}=\emptyset$. On the other hand,
by the previous argument, we know that $\pi(C)=\{2,3,5\}$. This implies that $M=1$ and so $C_0\cong {\rm SL}(2,5)$.
Also by Lemma~\ref{l2}, we know that this such Frobenius complement exists. Hence $G$ can be isomorphic to
a Frobenius group $K:C$, where $K$ is a $7$-subgroup and $C$ contains a subgroup isomorphic to ${\rm SL}(2,5)$ whose index is at most $2$, Therefore if $G$ is a Frobenius group, then we get Case (1).

Now we assume that $G$ is neither Frobenius nor $2$-Frobenius group. Hence by Lemma~\ref{wil}, there exists a nonabelian simple group $S$ such that:
$$S\leq \overline{G}:=G/K\leq \textrm{Aut}(S)$$
in which $K$ is the Fitting subgroup of $G$. Since $\{2,7\}$ is an independent subset of $\ga(G)$, by Lemma~\ref{wil}, we conclude that $7\in\pi(S)$ and $7\not\in \pi(K)\cup\pi(\bar{G}/S)$. Also we know that $\pi(S)\subseteq\pi(G)$.
Since $\pi(G)=\{2,3,5,7\}$, so by \cite[Table 8]{mogh}, we get that $S$ is isomorphic to $A_7$, $A_8$, $A_9$, $A_{10}$, $S_6(2)$, $O^+_8(2)$, $L_3(2^2)$, $L_2(2^3)$, $U_3(3)$, $U_4(3)$, $U_3(5)$, $L_2(7)$, $S_4(7)$, $L_2(7^2)$ or $J_2$.
Now we consider each possibility for the simple group $S$.

Let $S\cong L_2(7)$. Then $5\in \pi(K)$, since $5\not\in(\pi(S)\cup\pi(\bar{G}/S))$. On the other hand $S$ contains a $\{3,7\}$-subgroup $H$. Hence $G$ has a subgroup isomorphic to $K_5:H$ where $K_5$ is $5$-group. On the other hand $K_5:H$ is solvable and so there is an edge between to prime numbers in $\ga(K_5:H)$, which is impossible since $\ga(K_5:H)$ is a subgraph of $\ga(G)$. Thus $S\not\cong L_2(7)$.

Let $S\cong L_2(2^3)$. In this case, $5\in\pi(K)$. Also we know that $S$ contains a Frobenius group isomorphic to $8:7$. Hence by Lemma~\ref{maza}, we get that $G$ has an element order $5\cdot7$, which is a contradiction.

Let $S\cong A_8$, $A_9$ or $A_{10}$. Thus $S$ consists an element of order $3\cdot5$, which contradicts to the prime graph of $G$.

Let $S\cong J_2$, $O^+_8(2)$ or $S_6(2)$. In this case $S$ contains an element of order $15$, which is a contradiction.

By Lemma~\ref{l3}, the finite group $S$ can be isomorphic to each simple group $A_7$, $U_3(3)$, $U_4(3)$ and $U_3(5)$.

Let $S$ be isomorphic to ${\rm PSL}_2(49)$. Hence ${\rm PSL}_2(49)\leq \bar{G}\leq {\rm Aut}({\rm PSL}_2(49))$.

Let $\pi(K)$ contains a prime $r$ such that $r\not=7$. Since $K$ is nilpotent, we may assume that $K$ is a vector space over a field with $r$ elements (analogous to the proof of Lemma~\ref{mainl1}). Hence the prime graph of the semidirect product $K\rtimes {\rm PSL}_2(49)$ is a subgraph of $\ga(G)$.
Let $B$ be a Sylow $7$-subgroup of ${\rm PSL}_2(49)$. We know that $B$ is not cyclic. On the other hand $K\rtimes B$ is a Frobenius group such that $\pi(K\rtimes B)=\{r,7\}$. Hence $B$ should be cyclic which is a contradiction. This implies that $K=1$, since $7\not\in\pi(K)$.

We know that ${\rm Aut}({\rm PSL}_2(49))\cong Z_2\times Z_2$. Since in the prime graph of ${\rm PSL}_2(49)$ there is not any edge between $7$ and $2$, we get that $G\not\cong {\rm PSL}_2(49)$. Also if $G={\rm PSL}_2(49):\langle \theta\rangle$, where
$\theta$ is a field automorphism, then we get that $2$ and $7$ are adjacent in $G$, which is a contradiction.
If $G={\rm PSL}_2(49):\langle \gamma\rangle$, where
$\gamma$ is a diagonal-field automorphism, then we get that $G$ does not contain any element with order $2\cdot7$ (see \cite[Lemm 12]{pgl-spec}), which is contradiction, since in $\ga(G)$, $2\sim 7$. This argument shows that $G\cong {\rm PSL}_2(49)$, which completes the proof.
\end{proof}
\begin{problem}
Let $G=\textrm{PGL}(2, q)$ be an almost simple group related to the simple group $\textrm{PSL}(2, q)$. Find all
Frobenius group $H$ such that $\ga(H)=\ga(G)$.
\end{problem}

\end{document}